\documentclass{ecgd-l}

\usepackage{mathptmx}
\usepackage{enumerate}
\usepackage{amssymb}
\usepackage{amsmath}
\usepackage{amsthm}

\newcommand{\CC}{\mathbb C}
\newcommand{\DD}{\mathbb D}

\newcommand{\RR}{\mathbb R}

\renewcommand{\hat}{\widehat}
\renewcommand{\tilde}{\widetilde}

\newtheorem{theorem}{Theorem}[section]
\newtheorem{corollary}[theorem]{Corollary}
\newtheorem{lemma}[theorem]{Lemma}

\theoremstyle{remark}

\numberwithin{equation}{section}

\begin{document}

\title{Analytic capacity and holomorphic motions}

\author[S. Pouliasis]{Stamatis Pouliasis}
\address{Department of Mathematics and Statistics, Texas Tech University, Lubbock, TX 79409, USA}
\email{stamatis.pouliasis@ttu.edu}

\author[T. Ransford]{Thomas Ransford}
\address{D\'epartement de math\'ematiques et de statistique, Universit\'e Laval, Qu\'ebec (Qu\'ebec),  G1V 0A6, Canada.}
\email{thomas.ransford@mat.ulaval.ca}

\author[M. Younsi]{Malik Younsi}
\address{Department of Mathematics, University of Hawaii Manoa, Honolulu, HI 96822, USA}
\email{malik.younsi@gmail.com}

\subjclass[2010]{Primary 30C85, Secondary 31A15, 37F99}
%\date{6 Sept 2018}
\keywords{Analytic capacity, holomorphic motion, harmonic function.}
\thanks{The second author was supported by grants from NSERC and the Canada Research Chairs program.
The third author was supported by NSF Grant DMS-1758295}

\begin{abstract}
We study the behavior of the analytic capacity of a compact set under
deformations obtained by families of conformal maps depending holomorphically
on the complex parameter. 
We show that, under those deformations, the logarithm of the analytic capacity varies harmonically.
We also show that the hypotheses in this result cannot be substantially weakened.
\end{abstract}

\maketitle

\section{Introduction and statement of results}

The \emph{analytic capacity} of a compact set $K\subset\CC$ is defined by
$$\gamma(K):=\sup_{g}|g'(\infty)|,$$
where the supremum is taken over all holomorphic functions
$g:\hat{\CC}\setminus K\to\DD$.
Here $\DD$ denotes the open unit disk,
and  $g'(\infty)$ is defined by
\[
g'(\infty):=\lim_{z\to\infty}z(g(z)-g(\infty)).
\]
Analytic capacity was introduced by Ahlfors \cite{Ah47} in connection with
the Painlev\'e problem of characterizing removable singularities for
bounded holomorphic functions. For more information
on this subject, see the books of Garnett  \cite{Ga72}  and Tolsa  \cite{To14}.

The precise value of the analytic capacity is known only for
a relatively small class of compact sets, satisfying rather restrictive
geometric or connectivity properties. This  leads us to seek  techniques
for estimating  analytic capacity. 
One such technique was developed in  \cite{YR13},
using numerical methods.
In this note we obtain estimates via
a result on the variation of the analytic capacity
of a compact set that depends holomorphically on a parameter.

The prototype for our results is an old result from a paper of Yamaguchi \cite{Ya73}
concerning the  logarithmic capacity $c(K)$.
Yamaguchi's result states that, 
if $\lambda\mapsto K_\lambda$ is an analytic multifunction defined on a 
domain $D$, then $\lambda\mapsto\log c(K_\lambda)$ is a subharmonic function on~$D$.
For a brief introduction to analytic multifunctions, see Chapter~VII of Aupetit's book \cite{Au91}.
Yamaguchi's theorem is proved in \cite[Theorem~7.1.3]{Au91}.

It turns out that the analogous
result for analytic capacity is false. 
We do not stop here to give an example, since we shall establish a better result
in Theorem~\ref{T:thm2} below.
Thus, in order to find a substitute of  Yamaguchi's result for analytic capacity, 
we need to consider a more restricted notion of holomorphic variation of sets.

The appropriate notion is that of a holomorphic motion.
Given a subset $A$ of the Riemann sphere $\hat{\CC}$,
a \emph{holomorphic motion} of $A$ is a map
$f:\DD\times A\mapsto\hat{\CC}$
such that:
\begin{enumerate}[(i)]
\item for each fixed $z\in A$, the map $\lambda\mapsto f(\lambda,z)$
is holomorphic on $\DD$,
\item for each fixed $\lambda\in\DD$, the map $z\mapsto f(\lambda,z)$ is 
injective on $A$,
\item$f(0,z)=z$ for all $z\in A$.
\end{enumerate}

It is a remarkable fact, first established by  S\l odkowski  in \cite{Sl91}, 
that every holomorphic motion $f:\DD\times A\to\hat{\CC}$ 
admits an extension to map $f:\DD\times\hat{\CC}\to\hat{\CC}$
that is a holomorphic motion of $\hat{\CC}$.
For another proof of this result, 
as well as more background on holomorphic motions, see \cite{AM01}. 
Though we do not use this theorem directly, 
it does serve to motivate our consideration of  holomorphic motions of $\hat{\CC}$.

In what follows, we write $f_\lambda(z):=f(\lambda,z)$.
The following theorem is our first result.

\begin{theorem}\label{T:thm1}
Let $K$ be a compact subset of $\CC$ such that $\gamma(K)>0$.
Let $f:\DD\times \hat{\CC}\to\hat{\CC}$
be a holomorphic motion  such that, for each $\lambda\in\DD$, the map
$f_\lambda$ is holomorphic
on $\hat{\CC}\setminus K$  and satisfies $f_\lambda(\infty)=\infty$.
Then, writing $K_\lambda:=f_\lambda(K)$, we have that $\lambda\mapsto\log \gamma(K_\lambda)$ is a harmonic function on $\DD$.
\end{theorem}

Combining this result with Harnack's inequality for positive harmonic functions, 
we immediately obtain the following two-sided estimate for the analytic capacity of $K_\lambda$.

\begin{corollary}\label{C:thm1}
Assume, in addition, that $\gamma(K_\lambda)\le M$  for all $\lambda\in\DD$.
Then
\[
\frac{1-|\lambda|}{1+|\lambda|}
\le \frac{\log (M/\gamma(K_\lambda))}{\log (M/\gamma(K))}
\le \frac{1+|\lambda|}{1-|\lambda|}
\quad(\lambda\in\DD).\qed
\]
\end{corollary}

Theorem~\ref{T:thm1} yields a stronger conclusion than Yamaguchi's theorem 
(harmonic versus subharmonic), but it also requires a much stronger hypothesis.
It is natural to ask whether the hypothesis can be weakened. In particular,
is it possible to omit the assumption that $f_\lambda$ be holomorphic
on $\hat{\CC}\setminus K$? Our second result answers this question in the negative.

\begin{theorem}\label{T:thm2}
There exist a holomorphic motion $f:\DD\times\hat{\CC}\to\hat{\CC}$, satisfying 
$f_\lambda(\infty)=\infty$ for all $\lambda\in\DD$, and a compact subset $K$ of $\CC$
with $\gamma(K)>0$, such that, if we set $K_\lambda:=f_\lambda(K)$, 
then the functions $\lambda\mapsto\gamma(K_\lambda)$ and  $\lambda\mapsto\log \gamma(K_\lambda)$ 
are neither subharmonic nor superharmonic on $\DD$.
\end{theorem}

%%%%%%%%%%%%%%%%%%%%%%%%%%%%%%%%%%%
\section{Proofs}

For the proof of Theorem~\ref{T:thm1}, we shall need two lemmas.
The first  is part of the so-called $\lambda$-lemma, 
due to Ma\~n\'e, Sad and Sullivan \cite[p.193]{MSS83}.

\begin{lemma}\label{L:MSS}
A holomorphic motion $f:\DD\times A\to\hat{\CC}$ is jointly continuous in $(\lambda,z)$.\qed
\end{lemma}

The second lemma is a simple result about 
how the analytic capacity of a compact set 
transforms under conformal mapping of the complement.

\begin{lemma}\label{L:conf}
Let $K$ and $L$ be compact subsets of $\CC$ 
and let $h:\hat{\CC}\setminus K\to\hat{\CC}\setminus L$ 
be a conformal mapping such that $h(\infty)=\infty$. 
If $\gamma(K)>0$, then also $\gamma(L)>0$ and, for all $R>\max_{z\in K}|z|$, 
\[
\frac{\gamma(K)}{\gamma(L)}
=\Bigl|\frac{1}{2\pi i}\int_{|z|=R} \frac{h(z)}{z^2}\,dz\Bigr|.
\]
\end{lemma}

\begin{proof}
Under the hypotheses on $h$, we have $h(z)=az+O(1)$ as $|z|\to\infty$, 
where $a\in\CC\setminus\{0\}$. Given a holomorphic function 
$g:\hat{\CC}\setminus L\to\DD$, 
the composition $g\circ h$ is a holomorphic map 
from $\hat{\CC}\setminus K$ to $\DD$ with $(g\circ h)'(\infty)=ag'(\infty)$. Hence 
\[
|g'(\infty)|=|(g\circ h)'(\infty)|/|a|\le \gamma(K)/|a|.
\]
Taking the supremum over all such $g$, we deduce that $\gamma(L)\le \gamma(K)/|a|$. 

Applying the same argument to the inverse map
$h^{-1}:\hat{\CC}\setminus L\to\hat{\CC}\setminus K$, which satisfies
$h^{-1}(z)=a^{-1}z+O(1)$ as $|z|\to\infty$, we obtain $\gamma(K)\le\gamma(L)|a|$, 
and hence $\gamma(K)/\gamma(L)=|a|$.

Finally, to evaluate $a$, we observe that, by Cauchy's theorem, if $R>\max_{z\in K}|z|$, then
\[
\int_{|z|=R} \frac{h(z)}{z^2}\,dz
=\int_{|z|=R'} \frac{h(z)}{z^2}\,dz \quad(R'>R),
\]
and hence
\[
\int_{|z|=R} \frac{h(z)}{z^2}\,dz
=\lim_{R'\to\infty}\int_{|z|=R'} \frac{h(z)}{z^2}\,dz
=\lim_{R'\to\infty}\int_{|z|=R'} \frac{az+O(1)}{z^2}\,dz
=2\pi ia.\]
The result follows.
\end{proof}

\begin{proof}[Proof of Theorem \ref{T:thm1}.] 
Fix $R>\max_{z\in K}|z|$. By Lemma~\ref{L:conf}, 
applied to the conformal  mapping $f_\lambda:\hat{\CC}\setminus K\to\hat{\CC}\setminus K_\lambda$,
we have $\gamma(K_\lambda)>0$ for all $\lambda\in\DD$ and
\begin{equation}\label{E:thm1}
\frac{\gamma(K)}{\gamma(K_\lambda)}
=\Bigl|\frac{1}{2\pi i}\int_{|z|=R} \frac{f(\lambda, z)}{z^2}\,dz\Bigr|
\quad(\lambda\in\DD).
\end{equation}
By Lemma~\ref{L:MSS}, the map $(\lambda,z)\mapsto f(\lambda,z)$
is continuous. Also,  it is holomorphic in $\lambda$ (and finite-valued) for each fixed $z$
with $|z|=R$. It follows easily that the integral in \eqref{E:thm1} is a holomorphic function of $\lambda$. 
Since the integral  does not take the value zero, the log of its modulus is a harmonic function.
It follows that $\log\gamma(K_\lambda)$ is a harmonic function of $\lambda$.
\end{proof}

We now turn to Theorem~\ref{T:thm2}. For this, we need the following result of Astala \cite{As94}.
Here and in what follows, $\dim_H$ denotes the Hausdorff dimension.

\begin{lemma}\label{L:thm2}
Given $t\in(0,2)$, there exist a holomorphic motion $f:\DD\times\hat{\CC}\to\hat{\CC}$ satisfying $f_\lambda(\infty)=\infty$ 
for all $\lambda\in\DD$, and a compact subset $K$ of $\CC$, such that, writing $K_\lambda:=f_\lambda(K)$, we have
\begin{equation}\label{E:HD}
\dim_H(K_\lambda)=\frac{2t}{t+(2-t)(1-\lambda)/(1+\lambda)} 
\quad(0\le \lambda<1).
\end{equation}
\end{lemma}

\begin{proof}
Essentially this is proved in \cite[p.54]{As94}. In fact, what is shown there is that, 
given a sequence of pairwise disjoint disks $(D_k)_{k\ge1}$ inside the unit disk,  
there exist a holomorphic motion  $f:\DD\times\hat{\CC}\to\hat{\CC}$ 
and compact sets $(J_k)_{k\ge1}$ such that $f_\lambda(J_k)\subset D_k$ for all $k\ge1$ and all $\lambda\in\DD$, and
\[
\dim_H(f_\lambda(\cup_k J_k))=\frac{2t}{t+(2-t)(1-\lambda)/(1+\lambda)} 
\quad(0\le \lambda<1).
\]
It is easy to see that, in addition, $f$ may be chosen so that $f_\lambda(\infty)=\infty$ for all $\lambda\in\DD$.
If we further stipulate that the disks $D_k$ accumulate only at $0$, 
then $\overline{\cup_kJ_k}=\cup_kJ_k\cup\{0\}$
and $f_\lambda(0)=0$ for all $\lambda\in\DD$.
Thus, setting $K:= \overline{\cup_kJ_k}$ and $K_\lambda:=f_\lambda(K)$, 
we have $K_\lambda=f_\lambda(\cup_kJ_k)\cup\{0\}$ for all $\lambda\in\DD$.
Since the addition of a single point does not affect the Hausdorff dimension, it follows that \eqref{E:HD} holds.
\end{proof}

\begin{proof}[Proof of Theorem~\ref{T:thm2}]
Fix $t\in(0,1)$ and choose $f$ and $K$ as in Lemma~\ref{L:thm2}.
Then there exists $\delta\in(0,1)$ such that
$\dim_H(K_\lambda)<1$ for $\lambda\in[0,\delta)$ and $\dim_H(K_\lambda)>1$ for $\lambda\in(\delta,1)$.
Now it is a well-known property of analytic capacity of compact sets that $\dim_H<1$ implies $\gamma=0$
and that $\dim_H>1$ implies $\gamma>0$ (see for example \cite[p.34]{To14}). Thus we have
\[
\gamma(K_\lambda)
\begin{cases}
=0, &0\le \lambda<\delta,\\
>0, &\delta<\lambda<1.
\end{cases}
\]

In particular $\log\gamma(K_\lambda)=-\infty$ on $[0,\delta)$.
This straightaway rules out the possibility that $\log\gamma(K_\lambda)$ be superharmonic,
since superharmonic functions never take the value $-\infty$. 

It also shows that $\log\gamma(K_\lambda)$
cannot be subharmonic on $\DD$, because a subharmonic function that takes the value $-\infty$ on a line segment 
must be equal to $-\infty$ everywhere in its domain (see for example \cite[Exercise~2.5.1]{Ra95}),
and in our case  $\log \gamma(K_\lambda)>-\infty$ if $\lambda\in(\delta,1)$.

It is also easy to see that $\lambda\mapsto\gamma(K_\lambda)$ is not superharmonic on $\DD$.
Indeed, it attains a minimum without being constant, thus violating the minimum principle for
superharmonic functions. 

To treat the question of whether $\lambda\mapsto\gamma(K_\lambda)$ is
subharmonic, we invoke the following criterion due to Rad\'o
\cite[\S3.12]{Ra71}: given a non-negative function $u(\lambda)$, then $\log u(\lambda)$ is  subharmonic if and only if $|e^{\alpha\lambda}|u(\lambda)$ is subharmonic for each $\alpha\in\RR$.
Since we know that $\log\gamma(K_\lambda)$ is not subharmonic, 
by Rad\'o's criterion there exists $\alpha\in\RR$ such that 
$|e^{\alpha\lambda}|\gamma(K_\lambda)$ is not subharmonic.
Thus, if we replace $f(\lambda,z)$ by the holomorphic motion $e^{\alpha\lambda}f(\lambda,z)$,
which has the effect of replacing $K_\lambda$ by $e^{\alpha\lambda}K_\lambda$,
we obtain an example for which, in addition to all the other properties already established,
 $\lambda\mapsto\gamma(K_\lambda)$ is not subharmonic on~$\DD$.

This nearly proves the theorem. The only item lacking is that $\gamma(K)=0$, instead of $\gamma(K)>0$
as promised. To get round this, it is enough to change the base point of the holomorphic motion, as follows.
Fix $\lambda_0\in(\delta,1)$ and define
\[
\tilde{f}(\lambda, z):=f\Bigl(\frac{\lambda_0-\lambda}{1-\lambda\lambda_0}, f_{\lambda_0}^{-1}(z)\Bigr)
\quad\text{and}\quad
\tilde{K}:=K_{\lambda_0}.
\]
Then $\tilde{f}$ is a holomorphic motion, and $\tilde{K}_\lambda:=\tilde{f}_\lambda(\tilde{K})=K_{(\lambda_0-\lambda)/(1-\lambda\lambda_0)}$ for all $\lambda\in\DD$.
Thus the modified pair $\tilde{f},\tilde{K}$ satisfies all the requirements of the theorem.
\end{proof}

\end{document}